\documentclass[letter,12pt]{article}

\usepackage{amssymb}
\usepackage{amsthm}
\usepackage{amsfonts}
\usepackage[latin2]{inputenc}
\usepackage{amsmath}
\usepackage{anysize}
\usepackage{url}
\usepackage{titlesec}
\titleformat{\section}
  {\normalfont\large\bfseries}
  {\thesection}{1em}{}
\marginsize{2.6cm}{2.6cm}{1cm}{2cm}

\newtheorem{theorem}{Theorem}[section]
\newtheorem{lemma}{Lemma}[section]
\numberwithin{equation}{section}
\begin{document}
\title{{Morse Theory and Geodesics in the Space of K\"ahler Metrics}}
\author{Tam\'as Darvas\thanks{Research supported by NSF grant DMS1162070 and the Purdue Research Foundation.}}
\date{\vspace{0.1in} \emph{\small{To my parents.}}\vspace{-0.2in}}
\maketitle
\begin{abstract}Given a compact K\"ahler manifold $(X,\omega_0)$ let $\mathcal H_{0}$ be the set of K\"ahler forms cohomologous to $\omega_0$. As observed by Mabuchi \cite{m}, this space has the structure of an infinite dimensional Riemannian manifold, if one identifies it with a totally geodesic subspace of $\mathcal H$, the set of K\"ahler potentials of $\omega_0$. Following Donaldson's \cite{d1} research program, existence and regularity of geodesics in this space is of fundamental interest. In this paper, supposing enough regularity of a geodesic   $u:[0,1]\to \mathcal H$, connecting $u_0 \in \mathcal H$ with $u_1 \in \mathcal H$, we establish a Morse theoretic result relating the critical points of  $u_1-u_0$ to the critical points of $\dot u_0 = du/dt|_{t=0}$. As an application of this result, we prove that on all  K\"ahler manifolds, connecting K\"ahler potentials with smooth geodesics is not possible in general. In particular, in the case $X \neq \Bbb C P^1$, we will also prove that the set of pairs of potentials that can not be connected with smooth geodesics has nonempty interior. This is an improvement upon the findings of \cite{lv} and \cite{dl}.
\end{abstract}
\vspace{0.2in}
\section{Introduction}
Let $(X^n,\omega_0)$ be a connected compact K\"ahler manifold. The space of smooth K\"ahler potentials of $\omega_0$ is the set
\[\mathcal{H} := \left\{ v \in C^\infty(X) : \omega_0 + i \partial \overline{\partial} v > 0\right\}.\]
It is clear that $\mathcal{H}$ is a Fr\'{e}chet manifold as it is an open subset of $C^{\infty}(X)$. From this, it also follows that for each $v \in \mathcal{H}$ we can identify $T_v \mathcal{H}$ with $C^{\infty}(X)$. Following Mabuchi, we define a Riemannian metric on $T_v \mathcal{H}$:
\[\langle \xi,  \eta \rangle := \int_X \xi \eta (\omega_0 + i\partial \overline{\partial}v)^n, \ \ \ \ \xi,\eta \in T_v \mathcal{H}. \]

If $\phi:[0,1] \to \mathcal{H}$ is a smooth curve, one can think of  $\phi$ as an element of $C^{\infty}([0,1]\times X)$. Also, suppose that $\psi$ is a tangent vector field along $\phi$, which again can be treated as an element of $C^{\infty}([0,1]\times X)$.
The Levi-Civita connection of the metric introduced above is as follows:
\[\nabla_{\dot{\phi}} \psi := \dot{\psi} -
\frac{1}{2} \langle \nabla\psi, \nabla\dot\phi \rangle.\]
Hence a smooth curve $g:[0,1] \to \mathcal{H}$ is a geodesic if we have:
$$
0 = \nabla_{\dot{g}} \dot g = \ddot{g} -
\frac{1}{2} \langle \nabla\dot g, \nabla\dot g\rangle.$$

Let us now relate the above discussion to the space of K\"{a}hler metrics over $X$. Hodge theory implies that the correspondence $u \longrightarrow \omega_0 + i \partial \overline\partial{u}$ gives an onto map from $\mathcal{H}$ to $\mathcal{H}_0$. As found by Mabuchi, $\mathcal{H}_0$ can be identified with a totally geodesic subspace of $\mathcal{H}$. This correspondence will automatically put a Riemannain structure on $\mathcal{H}_0$. Moreover, one can prove that with this structure, $\mathcal{H}$ is isometric to the Riemannian product $\mathcal{H}_0 \times \Bbb R$. Hence to study the question of existence of geodesics in $\mathcal{H}_0$ one can study the same question in $\mathcal{H}$.

As discovered by Semmes \cite{s}, the geodesic equation can be rewritten as a complex Monge-Amp\`ere equation. Let $S = \left\{s \in \Bbb C: 0 <\textup{Im}s<1 \right\}$, and let $\omega$ be the pullback of the K\"{a}hler form $\omega_0$ to $\overline{S} \times X$ by the projection $\pi_X: S \times X \to X$. If $\phi:[0,1] \to \mathcal{H}$ is a smooth curve, let $u \in C^{\infty}(\overline{S} \times X)$ be defined by $u(s,x) := \phi(\textup{Im}s, x)$. Semmes observed that $\phi$ is a geodesic if and only if the following equation is satisfied for $u$:
\[(\omega + i \partial \overline{\partial}u)^{n+1}=0.\]
Since the form $\omega + \partial \overline{\partial}u$ is non-degenerate on each $X$-fiber, the above equation is the same as writing $(\omega + i \partial \overline{\partial}u)^{n+1}=0$ but $(\omega + i \partial \overline{\partial}u)^n \neq 0$. Summing up, we can say that if there exists a geodesic connecting $u_0 \in \mathcal{H}$ with $u_1 \in \mathcal{H}$ then there exists a smooth function $u \in C^{\infty}(\overline{S} \times X)$ that satisfies the following conditions:
\begin{alignat}{2}\label{bvp}
  &(\omega + i \partial \overline{\partial}u)^{n+1}=0 \nonumber\\
  &(\omega + i \partial \overline{\partial}u)^n \neq 0 \nonumber \\
  &u(r+ it,x) =u(it,x) \ \forall x \in X, t \in [0,1], r \in \Bbb R \\
  &u(0,x)=u_0(x), \ u(i,x) =u_1(x) \  \forall x \in X\nonumber.
\end{alignat}
Solutions to the above problem have been intensely studied. Uniqueness has already been found by Donaldson \cite{d1}. Chen \cite{c} proved that one always has a $\omega-$plurisubharmonic solution to (\ref{bvp}) in the sense of Bedford and Taylor \cite{bt} for which the current $\partial \overline{\partial}u$ is represented by a bounded form. Of course in this setting the condition $(\omega + i \partial \overline{\partial}u)^n \neq 0$ does not make any sense and hence it is omitted from the problem.

In \cite{lv} the authors proved that on certain K\"ahler manifolds admitting special symmetries there exist boundary data for which (\ref{bvp}) has no $C^3$ solutions. This provided the first example showing that, using smooth geodesics, connecting points in $\mathcal H$ is not possible in general. In \cite{dl}, under similar circumstances, it was shown that the regularity result of Chen we just mentioned is sharp.

In this paper we only investigate $C^3$ solutions of problem (\ref{bvp}). More specifically, with the above conventions in mind, we prove the following result:
\begin{theorem}\label{main}
Suppose $u$ is a $C^3$ solution of (\ref{bvp}). Then all the critical points of $ \dot u_0 := du(it, \cdot )/dt|_{t=0} \in C^2(X)$ are among the critical points of $u_1 - u_0$. Moreover, if $u_1 - u_0$ is a Morse function, then $\dot u_0$ is a Morse function as well.
\end{theorem}

The above result establishes a connection between the boundary data and initial tangent vector of $C^3$ solutions of the complex Monge-Amp\`ere equation. This theorem is proved by a careful analysis of the first order variation of the Monge-Amp\`ere foliation of $u$ at each critical point of the initial tangent vector $\dot u_0$. A similar method has been applied with success in \cite{lv} as well. As an application we prove the following theorem:

\begin{theorem}\label{nogeo}
For each  connected compact K\"ahler manifold ($X, \omega_0$)  one can find a pair of potentials $u_0 ,u_1 \in \mathcal H$ that can not be connected by a smooth geodesic. If $X \neq \Bbb C P^1$, then the set of such pairs of potentials has nonempty interior in $C^\infty(X) \times C^\infty(X)$.
\end{theorem}

This result is an improvement upon the similar non-existence result of \cite{lv}. The examples constructed there only work on manifolds admitting holomorphic isometries $A:X \to X$ having an isolated fixed point and satisfying the identity $A^2 = Id$.

In K\"ahler geometry, the special role of symmetries, especially that of one-parameter group actions, has a long history. The first result of this kind was that of Matushima \cite{m}, who proved that if the Lie algebra of holomorphic vector fields on a complex manifold is not reductive, then the manifold admits no K\"ahler-Einstein metrics. By contrast, in the absence of holomorphic vector fields, the search for canonical K\"ahler metrics becomes more promising. Calabi \cite{ca} observed that in this case, all extremal metrics will have constant scalar curvature. The Futaki invariant, that serves as an obstruction to finding constant scalar curvature metrics, vanishes in this situation as well \cite{f}.

With the above picture and the examples of \cite{lv} in mind, one might perhaps think that constructing pairs of K\"ahler potentials that can not be connected by a smooth geodesic requires the presence of holomorphic symmetries. The interest of Theorem \ref{nogeo} is showing that this is not the case.

It would be interesting to see if one can prove the last statement of the above result for $\Bbb C P^1$ as well. We summarize the proof of Theorem \ref{nogeo} here. In the case $X = \Bbb C P^1$ every K\"ahler metric $\omega_0$ is in the K\"ahler class of some  multiple of the Fubini-Study metric $\omega_{FS}$. However, each metric $c \omega_{FS}$ satisfies the symmetry conditions of Theorem 1.1 in \cite{lv}, and so this theorem provides the $u_0,u_1 \in \mathcal H$ that we seek. Hence, to prove the result, we only have to consider the case $X \neq \Bbb C P^1$. The rest of the argument  will be a consequence of the following more precise statement:

\begin{theorem}\label{noc3}
Suppose $X \neq \Bbb C P^1$. For each connected compact K\"ahler manifold $(X, \omega_0)$ one can find boundary data $u_0 ,u_1 \in \mathcal H$ for which problem (\ref{bvp}) has no $C^3$ solutions. In particular, the set of such pairs has nonempty interior in $\mathcal H$.
\end{theorem}
\emph{Acknowledgement.} I would like to thank Y. A. Rubinstein for bringing to my attention the paper \cite{rz} and for his patience in conversations related to the topic. As we found out later, he was developing similar ideas to ours at the time this paper was written. I would also like to thank L. Lempert for his guidance and for his many useful suggestions.
\vspace{0.2 in}
\section{The Monge-Amp\`ere foliation}

In this section we summarize known facts about the Monge-Amp\`ere foliation, first discusessed in \cite{bk}. For a more complete picture, the reader is referred to \cite{d2} or Section 2.1 in \cite{rz}.

The reason $C^3$ solutions of the complex Monge-Amp\`ere equation are special is that this is the minimal amount of regularity that is needed for the existence of the Monge-Amp\`ere foliation. This is defined as follows. Let $u \in C^\infty(\overline{S} \times X)$ be a solution of (\ref{bvp}). This means that for each $(s,x) = (r + i t, x) \in \overline{S} \times X$, $\textup{Ker}(\omega + i \partial \overline{\partial}u)|_{(s,x)}$ is a one dimensional complex subspace of $T_{(s,x)}(\overline{S} \times X)$, called the Monge-Amp\`ere distribution. Since the form $\omega + i \partial \overline{\partial}u$ is closed, this complex subbundle of $T(\overline{S} \times X)$ is integrable, hence gives rise to a foliation in which each leaf is a one dimensional complex submanifold of $\overline{S} \times X$.

We will denote by $\mathcal F_x$ the leaf of this foliation passing through the point $(0,x)$. Also, $u_t( \cdot )$ and $\dot u_t( \cdot  )$ will stand for $u(it, \cdot)$ and $du(it,\cdot)/dt$ for $t \in[0,1]$. We should note that  our formulation of (\ref{bvp}) allows solutions $u$ for which some of the forms $\omega_0 + i \partial_X \bar \partial_X u_t$ are degenerate on $X$. Because of this, in general, the leaves $\mathcal F_x$ can not be assumed to be diffeomorphic to the strip $\overline{S}.$ This will only be a minor inconvenience as we will see in a moment.

As observed by Semmes \cite{s}, the vector field ${\partial}/{\partial r} + J\nabla_{g_{0}}\dot u_{0}/2$ generates $\textup{Ker}(\omega + i \partial \overline{\partial}u)|_{(r,x)}$, where $g_{0}$ is the Riemannian metric corresponding to the K\"ahler metric $\omega_0 + i \partial_X \overline{\partial}_X u_0$ and $x \in X$, $r \in \Bbb R$. Since $J \nabla_{g_{0}}\dot u_{0}/2$ is independent of $r$, this has a very interesting consequence.  If $r \to f_x(r)$, $r \in \Bbb R$ is the trajectory of the time independent vector field $J\nabla_{g_{0}}\dot u_{0}/2$ with initial data $f_x(0)=x$, then the leaf $\mathcal F_x$ contains the image of the arc $r \to \Gamma_x(r)=(r,f_x(r))$. Hence, with a slight abuse of precision, one should think of each leaf $\mathcal F_x$ as the unique analytic extension of the curve $\Gamma_x$. This fact will be used in the next section.
\vspace{0.2 in}
\section{Proof of Theorem \ref{main}}

By replacing $\omega_0$ with $\omega_0 + i \partial \overline{\partial}u_0$ and $u(s,x)$ with $u(s,x) - u_0(x)$, we can assume that $u_0=0$. We prove Theorem \ref{main} in two steps.

\begin{lemma}\label{crit} If $x_0$ is a critical point of $\dot u_0$ then it is a critical point of $\dot u_t$ for all $t \in [0,1]$ and the leaf of the Monge-Amp\`ere foliation through $x_0$ is $\overline S \times \left\{x_0\right\}$. In particular $t \to u_t(x_0)$ is linear.
\end{lemma}

\begin{proof}
Denote $\omega_t = \omega_0 + i \partial_X \overline{\partial}_X u_t$. We saw in the preceding section that the vector ${\partial}/{\partial r} + J\nabla_{{g_0}}{\dot u_0}/2$ generates the Monge-Amp\`ere distribution at any point $(r, x) \in \Bbb R \times X$.

Since $x_0$ is a critical point of $\dot u_0$, the vector field $J\nabla_{{g_0}}{\dot u_0}/2$ vanishes at $x_0$, hence its trajectory $r \to f_{x_0}(r)$ is $x_0$. Therefore, following the reasoning of the preceding section, we obtain that $\mathcal F_{x_0} = \overline S \times \left\{{x_0}\right\}$. Hence, $\partial/ \partial r$ generates $\textup{Ker}(\omega + i \partial \overline{\partial}u)|_{(s,x_0)}$ for any $s \in \overline{S}$. In any holomorphic chart of $X$ around $x_0$, this implies that $u_{\overline{s}z_j}(s,x_0)=u_{s \overline{z}_j}(s,x_0)=0$ and $u_{s\overline{s}}(s,x_0)=0$ for all $s \in \overline{S}$ and for all $j=1,\dots, n$. As $u$ is translation invariant, $\dot u_t = 2i \partial u/ \partial s$, $\ddot u_t = 4 \partial^2 u/ \partial s\bar \partial s$ and we obtain the desired result.
\end{proof}

From this lemma, we obtain that if $x_0$ is a critical point of $\dot u_0$, then it is a critical point of $u_1$ as well, since in any holomorphic chart around $x_0$, we have ${u_1}_{z_j}(x_0) = \int_0^1 \dot {u_t}_{z_j}(x_0)dt=0 \textup{  and  } {u_1}_{\overline{z}_j}(x_0) = \int_0^1 \dot {u_t}_{\overline{z}_j}(x_0)dt=0$ for all $j=1,\dots, n$. All that is left, is to prove the nondegeneracy in Theorem \ref{main}. We will do this in the following lemma.

\begin{lemma}
If $x_0$ is a degenerate critical point of $\dot u_0$, then it is a degenerate critical point of $u_1$ as well. Moreover, for the Hessians of $\dot u_0$ and $u_1$ at $x_0$, we have the inclusion $\textup{Ker }H\dot u_0(x_0)  \subseteq \textup{Ker }H {u_1(x_0)}$.
\end{lemma}

\begin{proof} Most of the computations to follow are taken over from \cite[Section 2]{lv}. We extend the Hessian forms of both $u_1$ and $\dot u_0$ to complex bilinear forms.  Then we have the following representation of these forms at any point $x \in X$ in any complex coordinates:
\[ H{u_1}(x) = \left[ \begin{array}{cccc}
 {u_1}_{zz}(x)  & {u_1}_{z\overline{z}}(x)\\
{u_1}_{\overline{z}z}(x) & {u_1}_{\overline{z}\overline{z}}(x)\\\end{array} \right] \textup{, }  H\dot u_0(x) = \left[ \begin{array}{cccc}
 \dot {u_0}_{zz}(x)  & \dot {u_0}_{z\overline{z}}(x)\\
\dot {u_0}_{\overline{z}z}(x) & \dot {u_0}_{\overline{z}\overline{z}}(x)\\\end{array} \right],\]
where ${u_1}_{zz}(x) = \left\{{u_1}_{z_j z_k}(x)\right\}_{j,k = 1,n}$ and similarly for all the other terms.

We fix local coordinates $z_j$, $j = 1, . . . ,n$ on a neighborhood $V \subset X$ of $x_0$ such that $\omega_0 |_{x_0} = i dz_j \wedge d\overline{z}_j |_{x_0}$, the local coordinates map $V$ on a convex set in $\Bbb{C}^n$ and there exist a potential $w_0$ on $V$ such that $\omega_0 = i \partial \overline{\partial} w_0$ on $V$. We can clearly suppose that ${w_0}_{z_j z_k} |_{x_0} = 0$. We denote by $w$ the function $\pi^*_X w_0$. We identify $V$ with its image in $\Bbb{C}^n$ and $x_0$ with $0 \in \Bbb{C}^n$. Then $\overline{S} \times V$ is identified with a subset of $\overline{S} \times \Bbb{C}^n$.

From Lemma \ref{crit} it follows that $f_{x_0}=f_0 \equiv 0$ and  the leaf of the Monge Amp\`ere foliation passing through $(0,x_0)$ is $\overline{S}\times \left\{ x_0 \right\}$.  As discussed in Section 2 of \cite{lv}, this implies that for any $d >0$ and for small enough $x$, it makes sense to holomorphically extend the trajectories $r \to f_x(r)$  from the segment $[-d,d] \subset \Bbb R$ to the compact set $\left\{z\in   \overline{S}| -d \leq\textup{Re }z  \leq d\right\}$. Because of our regularity assumptions, we have that the leafs of the Monge-Amp\`ere foliation change differentiably. Hence for any $a \in \Bbb C^n$ one can define the functions $\varphi_j := {df_{(ta)}}_j/dt|_{t=0}$,  continuous on $\overline{S}$ with values in $\Bbb C$, that are holomorphic on $S$ for $j=1,\dots,n$. It follows from basic properties of the Monge-Amp\`ere foliation that whenever the maps below are defined, they are holomorphic:
$$
s \to {w}_{z_j}(s,f_{ta}(s)) + u_{z_j}(s,f_{ta}(s)), \textup{\hspace{0.2 in}}$$
$j = 1,...,n.$ Differentiating the above maps with respect to $t$ we obtain that the functions
$$
s \to \sum_{k=1}^n\Big(\left\{{w}_{z_j z_k}(s,0) + u_{z_j z_k}(s,0)\right\}\varphi_k(s)+ \left\{{w}_{z_j \overline{z}_k}(s,0) + u_{z_j \overline{z}_k}(s,0)\right\}\overline{\varphi_k(s)}\Big), \textup{\hspace{0.05 in}}$$
$j = 1,...,n, s \in S$, are holomorphic as well. From our choice of $w_0$ it follows that ${w}_{zz}(s,0) \equiv 0$ and $w_{z\overline{z}}(s,0)\equiv I$. After differentiating the above maps by $\partial/\partial \overline{s}$ as well, we find that
$$
\sum_{k=1}^n\Big(u_{z_j z_k \overline{s}}(s,0)\varphi_k(s)+ u_{z_j \overline{z}_k \overline{s}}(s,0)\overline{\varphi_k(s)} + \left\{\delta_{jk} + u_{z_j \overline{z}_k}(s,0)\right\}\overline{\varphi'_k(s)}\Big)=0, \textup{\hspace{0.2 in}} j = 1,...,n.
$$
Summing up, with the notation introduced at the beginning, we obtain that the function $\psi : S \to \Bbb C^n$ defined by
$$
\psi(s):=u_{zz}(s,0)\varphi(s)+ \left\{I + u_{z \overline{z}}(s,0)\right\}\overline{\varphi(s)},
$$
is holomorphic and
\begin{equation}\label{fol5}
\left\{I + u_{\overline{z}z}(s,0)\right\}\varphi'(s)=-u_{\overline{z}{z} s}(s,0)\varphi(s)-u_{\overline{z}\overline{z} s}(s,0)\overline{\varphi(s)}.
\end{equation}
 With the notation $P = I + u_{z\overline{z}}(i,0)=I + {u_1}_{z\overline{z}}(0)$, $Q = u_{zz}(i,0)={u_1}_{zz}(0)$, $A = - u_{\overline{z}z s}(0,0)$ and $B = -u_{\overline{z}\overline{z}s}(0,0)$ we find that
\begin{equation}\label{fol6}
\psi(z) =  \begin{cases}
 \overline{\varphi(s)} &\textup{ if Im } s = 0\\
P\overline{\varphi(s)} + Q \varphi(s) &\textup{ if  Im } s = 1\\
 \end{cases}.
 \end{equation}
Restricting (\ref{fol5}) to real $s$ we have
$$
\varphi'(s)=A\varphi(s) + B \overline{\varphi(s)}, \textup{ \hspace{0.2 in}} s \in \Bbb R .
$$
We also observe that $\varphi(0)=a$. Hence, to find the values of $\varphi$, we need to solve the following initial value problem on the real line :
$$
\left[ \begin{array}{cccc}\varphi'(s) \\ \overline{\varphi}'(s)\end{array}\right]=\left[ \begin{array}{cccc}
 A & B\\
\overline{B} & \overline{A}\\
 \end{array} \right]\left[ \begin{array}{cccc}\varphi(s) \\ \overline{\varphi}(s)\end{array}\right], \ \left[ \begin{array}{cccc}\varphi(0) \\ \overline{\varphi}(0)\end{array}\right]=\left[ \begin{array}{cccc}a \\ \overline{a}\end{array}\right]
$$
Since $A  = i \dot {u_0}_{\overline{z}z}(0)/2$, $B  = i \dot {u_0}_{\overline{z}\overline{z}}(0)/2$ and $H \dot u(0)$ is assumed to be degenerate, we obtain that the matrix of the above linear initial value problem is degenerate too. This means that there exists a nonzero $\alpha \in \Bbb C^n$ such that
$$\left[ \begin{array}{cccc}
 A \alpha  + B \overline{\alpha}\\
\overline{B}\alpha + \overline{A} \overline{\alpha}\\
 \end{array} \right] = 0
$$
Thus, for $a =  \alpha$, $\varphi(s) \equiv  \alpha$ is a solution of the initial value problem. Using ($\ref{fol6}$) and analytic continuation we obtain $\overline{\alpha} = P \overline{\alpha} + Q \alpha$. Hence ${u_1}_{z\overline{z}}(0)\overline{\alpha} + {u_1}_{zz}(0)\alpha = 0$. So the vector $[\alpha, \overline{\alpha}]^T$ is in the kernel of the Hessian of $u_1$ at $0$, therefore $x_0$ is a degenerate critical point of $u_1$. The inclusion $\textup{Ker }H\dot u_0(x_0)  \subseteq \textup{Ker }H {u_1(x_0)}$ clearly follows from our arguments.
\end{proof}
\section{Proof of Theorem \ref{noc3}}

We start out with a Morse function $w:X \to \Bbb R$ with a minimal amount of critical points and distinct critical values. The condition $X \neq \Bbb C P^1$ implies that $b_i(X) >0$ for some $0 < i < 2n$. From the Morse inequalities it follows that $v$ has at least one saddle point $x_0$. Let $p$ be the index of $w$ at $x_0$.

We choose a holomorphic coordinate chart $(z_1 = x_1 + iy_1,\ldots, z_n=x_n + i y_n)$ around the point $x_0$ such that $\omega_0|_{x_0} = \sum_{j=1}^n i dz_j \wedge d \bar z_j$. It is easy to see that one can choose a smooth diffeomorphism $\varphi$ of $X$, such that $\varphi(x_0)=x_0$ and in our fixed chart, the Hessian of the Morse function $v = w \circ \varphi$ at $x_0$  is diagonal with the following entries on the diagonal:
$$\frac{\partial^2 v}{\partial x_j^2}(x_0) = \frac{\partial^2 v}{\partial y_k^2}(x_0) = 1 \textup{ if }1 \leq j \leq n,p+1 \leq k \leq n \textup{ and }\frac{\partial^2 v}{\partial y_l^2}(x_0) = -1 \textup{ if }1 \leq l \leq p.$$

By multiplying $v$ with a small constant, we can assume that $v \in \mathcal H$. We clearly have $v_{z_1 \overline{z}_1}(x_0)=0$ and $|v_{z_1 z_1}(x_0)| \neq 0$. We can choose a number $\gamma > 0$ such that  $$|v_{z_1 z_1}(x_0)| >2 \gamma + v_{z_1 \overline{z}_1}(x_0) \textup{ and }\gamma + v_{z_1 \overline{z}_1}(x_0) > 0.$$
Using this and the fact that $v_{z \overline{z}}(x_0)$ is diagonal, by a lemma that will be provided below, one can find a smooth function $\rho$ supported inside our coordinate patch, that is identically equal to $1$ in a neighborhood of $x_0$ such that
$$\omega_0 + i \partial \overline{\partial}(\gamma-1)\rho(z)|z_1|^2 >0 \textup{ and }  (\omega_0 +  i \partial \overline{\partial}v) + i \partial \overline{\partial}(\gamma-1)\rho(z)|z_1|^2 > 0.$$

Assume that problem (\ref{bvp}) has a $C^3$ solution $u$ with boundary data $u_0 = (\gamma-1)\rho(z)|z_1|^2$ and $u_1 = v+(\gamma-1)\rho(z)|z_1|^2 $.  By Theorem \ref{main}, $\dot u_0$ is a Morse function, and all its critical point are
critical points of $v$. As $v$ has the minimal amount of critical
points, the critical sets of $\dot u_0$ and $v$ coincide, in particular $x_0$ is critical for $\dot u_0$.

Since $\textup{rank}(\omega + i\partial\bar\partial u)\equiv n$ and $\omega + i\partial\bar\partial u$ is semi-positive on the boundary of $\overline{S} \times X$, we obtain that $\omega + i\partial\bar\partial u \geq 0$, in particular $u$ is $\omega-$plurisubharmonic. It follows from this and Lemma \ref{crit} that all the conditions of Lemma 3.2 in \cite{dl} are satisfied for the K\"ahler metric $\omega' = \omega + i\partial \bar \partial (\gamma -1)\rho(z)|z_1|^2$ and the $\omega'-$plurisubharmonic function $u' = u - (\gamma -1)\rho(z)|z_1|^2$. However, the conclusion of this result is violated since for $v' = u'_1$ and $\xi' = \partial/\partial z_1$ we have:
\begin{equation}\label{viol}|\sum_{j,k}v'_{z_j z_k}(x_0)\xi'_j \xi'_k| >\sum_{j,k}(2 \omega'_{jk}\xi'_j \overline{\xi'}_k + v'_{z_j \overline{z_k}}(x_0)\xi'_j \overline{\xi'}_k).\end{equation}

To prove the last statement of the theorem, we refer to elements of  Morse theory. If one perturbs slightly $u_0$ and $u_1$ in the $C^\infty$ topology, it is a standard result that $u_1 - u_0$ is still a Morse function with distinct critical values. Also, by \cite[Theorem 2.2, p.79]{gg}, we know that for small perturbations, $u_1 - u_0$ will still have a minimal amount of critical points.  Lastly, if the perturbation is small enough, then $u_0,u_1 \in \mathcal H$, and at one of the critical points ($\ref{viol}$) still holds. This completes the main argument of the proof.

As promised, we  provide the following lemma initially found by Donaldson \cite[Lemma 8]{d2}, \cite[Lemma 3.3]{lv}. Since its original formulation is slightly inadequate for our purposes, we reformulate it here and give a proof for completeness.

\begin{lemma} Suppose $\omega^0$ and $\omega^1$ are two K\"ahler metrics on a complex manifold $X$. Let $q^0$ and $q^1$ be two real smooth functions on $X$ with critical points at $x_0$ and $q^0(x_0)= q^1(x_0)=0$. If in some holomorphic coordinate  patch $U$ around $x_0$ we have ${q^l}_{z\overline{z}}(x_0)(\xi,\overline{\xi}) + {\omega^l}(x_0)(\xi, \overline{\xi}) > 0$ for all nonzero $\xi \in T_{x_0} X$, $l=0,1$, then one can find a smooth function $\rho$ supported in $U$ that is identically $1$ in a neighborhood of $x_0$ such that $\omega^l + i \partial \overline{\partial}(\rho q^l) > 0$, $l=0,1$.
\end{lemma}
\begin{proof} By shrinking $U$ and possibly rescaling the coordinates, there exists $1 > \alpha > 0$ such that $\omega^l + i \partial \overline{\partial}q^l > \alpha\omega^l$ for $l=0,1$ on $U$ and the coordinates map $x_0$ to $0$ and $U$ to a neighborhood of $\left\{ |z| \leq 2\right\} \subset \Bbb C^n$.

Now take a smooth function $\beta: [-\infty,+\infty) \to [0, 1]$ such that $\beta(t) = 0$ if $t>0$ and $\beta(t)=1$ if $t< -1$. Then $\rho(t)=\beta(\varepsilon \log t )$, $\varepsilon >0$, is supported in $[0,1]$, $\rho(t)=1$ if $t < e^{-1/\varepsilon}$ and $t|\rho'(t)|,t^2 |\rho''(t)|< C\varepsilon$ for some constant $C$. From all of this it is easy to see that on $U$ we have
$$\omega^l + i\partial\overline{\partial}(\rho(|z|^2) q^l) > \omega^l + i \rho(|z|^2)\partial\overline{\partial}q^l - D \varepsilon \omega^l > (\alpha - D\varepsilon)\omega^l$$
for some constant $D$ independent of $\varepsilon$, $l = 0,1$. Since $\rho(|z|^2) q^l$ is identically $0$ outside $U$, for small enough $\varepsilon$, $\rho(|z|^2)$ has the required properties.
\end{proof}

\vspace{0.1 in}
\textsc{Department of Mathematics, Purdue University, West Lafayette, IN\ 47907}
\emph{E-mail address: }\texttt{\textbf{tdarvas@math.purdue.edu}}

\end{document}